\documentclass[11pt, reqno]{amsart}
\usepackage{latexsym}
\usepackage{amssymb}
\usepackage{mathrsfs}
\usepackage{amsmath}
\usepackage{fancybox,color}
\usepackage{enumerate}
\usepackage{hyperref}
\usepackage[latin1]{inputenc}
\usepackage{eurosym}

\newcommand{\kom}[1]{}
\renewcommand{\kom}[1]{{\bf [#1]}}

 \def\1{\raisebox{2pt}{\rm{$\chi$}}}

\newtheorem{theorem}{Theorem}[section]
\newtheorem{corollary}[theorem]{Corollary}
\newtheorem{lemma}[theorem]{Lemma}
\newtheorem{proposition}[theorem]{Proposition}
\newtheorem{definition}[theorem]{Definition}
\newtheorem{remark}[theorem]{Remark}

\newcommand{\R}{{\mathbb R}}

 \newcommand{\eps}{{\varepsilon}}
 \def\1{\raisebox{2pt}{\rm{$\chi$}}}
 
\def\vint_#1{\mathchoice%
          {\mathop{\kern 0.2em\vrule width 0.6em height 0.69678ex depth -0.58065ex
                  \kern -0.8em \intop}\nolimits_{\kern -0.4em#1}}%
          {\mathop{\kern 0.1em\vrule width 0.5em height 0.69678ex depth -0.60387ex
                  \kern -0.6em \intop}\nolimits_{#1}}%
          {\mathop{\kern 0.1em\vrule width 0.5em height 0.69678ex
              depth -0.60387ex
                  \kern -0.6em \intop}\nolimits_{#1}}%
          {\mathop{\kern 0.1em\vrule width 0.5em height 0.69678ex depth -0.60387ex
                  \kern -0.6em \intop}\nolimits_{#1}}}
\def\vintslides_#1{\mathchoice%
          {\mathop{\kern 0.1em\vrule width 0.5em height 0.697ex depth -0.581ex
                  \kern -0.6em \intop}\nolimits_{\kern -0.4em#1}}%
          {\mathop{\kern 0.1em\vrule width 0.3em height 0.697ex depth -0.604ex
                  \kern -0.4em \intop}\nolimits_{#1}}%
          {\mathop{\kern 0.1em\vrule width 0.3em height 0.697ex depth -0.604ex
                  \kern -0.4em \intop}\nolimits_{#1}}%
          {\mathop{\kern 0.1em\vrule width 0.3em height 0.697ex depth -0.604ex
                  \kern -0.4em \intop}\nolimits_{#1}}}

\newcommand{\aveint}[2]{\mathchoice%
          {\mathop{\kern 0.2em\vrule width 0.6em height 0.69678ex depth -0.58065ex
                  \kern -0.8em \intop}\nolimits_{\kern -0.45em#1}^{#2}}%
          {\mathop{\kern 0.1em\vrule width 0.5em height 0.69678ex depth -0.60387ex
                  \kern -0.6em \intop}\nolimits_{#1}^{#2}}%
          {\mathop{\kern 0.1em\vrule width 0.5em height 0.69678ex depth -0.60387ex
                  \kern -0.6em \intop}\nolimits_{#1}^{#2}}%
          {\mathop{\kern 0.1em\vrule width 0.5em height 0.69678ex depth -0.60387ex
                  \kern -0.6em \intop}\nolimits_{#1}^{#2}}}

\newcommand{\Om}{\Omega}

\newcommand{\dist}{\operatorname{dist}}

\begin{document}

\title[A game for the parabolic dominative $p$-Laplace equation]{A control problem related to the parabolic dominative $p$-Laplace equation}

\author[H\o eg]{Fredrik Arbo H\o eg}
\address{Department of Mathematical Sciences, Norwegian University of Science and Technology, NO-7491 Trondheim, Norway}
\email{fredrik.hoeg@ntnu.no}

\author[Ruosteenoja]{Eero Ruosteenoja}
\address{Department of Mathematical Sciences, Norwegian University of Science and Technology, NO-7491 Trondheim, Norway}
\email{eero.k.ruosteenoja@ntnu.no}

\date{\today}

\begin{abstract}
We show that value functions of a certain time-dependent control problem in $\Omega\times (0,T)$, with a continuous payoff $F$ on the parabolic boundary, converge uniformly to the viscosity solution of the parabolic dominative $p$-Laplace equation $$2(n+p)u_t=\Delta u+(p-2)\lambda_n(D^2 u),$$ with the boundary data $F$. Here $2< p< \infty$, and $\lambda_n(D^2 u)$ is the largest eigenvalue of the Hessian $D^2 u$.
\end{abstract}

\maketitle


\section{Introduction}\label{intro}
In this paper we give a control problem interpretation for the parabolic dominative $p$-Laplace equation
\begin{equation}\label{parab}
2(n+p)u_t=\mathcal{D}_pu\quad \quad \text{in}\quad \Omega_T.
\end{equation}
Here $\Omega_T:=\Omega\times (0,T),$ where $\Omega\subset \R^n$ is a bounded domain satisfying a uniform exterior sphere condition, and 
$$\mathcal{D}_pu:=(\lambda_1+...+\lambda_{n-1})+(p-1)\lambda_n=\Delta u+(p-2)\lambda_n,$$ where $2< p<\infty$, and $\lambda_1\leq\lambda_2\leq...\leq \lambda_n$ are the eigenvalues of the Hessian $D^2 u$. The operator $\mathcal{D}_p$ is called \emph{the dominative $p$-Laplacian}, introduced by Brustad \cite{Bru17,Bru18} and later studied by Brustad, Lindqvist and Manfredi \cite{BLM18} and H\o eg \cite{Hoe19} in the elliptic case. The dominative $p$-Laplacian explains the superposition principle of the $p$-Laplace equation, see \cite{CZ03,LM08} for more about this property. The operator $\mathcal{D}_p$ is sublinear, so it is convex, and equation \eqref{parab} is uniformly parabolic. By Theorem 3.2 in \cite{Wan92}, viscosity solutions of \eqref{parab} are in $C^{2+\alpha,\frac{2+\alpha}{2}}(\Omega_T)$ for some $\alpha>0$.

Let $u$ be a viscosity solution of \eqref{parab} with a given continuous boundary data $F$ on $\partial_p \Omega_T:=(\Omega\times \{0\})\cup (\partial \Omega\times [0,T]).$ By \cite{CIL92}, the solution is unique. In Section \ref{control} we see that for $\eps>0$ and the boundary data $F$, there is a unique Borel-measurable function $u_\eps$ satisfying a \emph{dynamic programming principle} (hereafter DPP) 
\begin{align}\label{dpp}
&u_\eps (x,t) = \frac{n+2}{p+n}\vint_{B_\eps(x)}u_\eps(y,t-\eps^2)\,dy\nonumber\\
&\quad +\frac{p-2}{p+n}\sup_{|\sigma|=1}\left[\frac{u_\eps(x+\eps\sigma,t-\eps^2)+u_\eps(x-\eps\sigma,t-\eps^2)}{2}\right] \quad \text{in} \, \, \Omega_T. 
\end{align}
Here $B_\eps(x)\subset \R^n$ is a ball centered at $x$ with the radius $\eps$, in the first term we have an average integral, and in the second term the supremum is taken over all unit vectors in $\R^n$. In Theorem \ref{convergence} we show that $u_\eps\rightarrow u$ uniformly when $\eps\rightarrow 0$. The idea of the proof is to first show that the family $\{u_\eps\}_{\eps>0}$ is uniformly bounded and asymptotically equicontinuous, and use a variant of the Arzel\'a-Ascoli theorem to see that solutions of the DPP converge uniformly to some continuous function. To show that the uniform limit is the viscosity solution of \eqref{parab}, we make use of an asymptotic mean value formula 
\begin{align}\label{asy}
&\frac{n+2}{p+n}\vint_{B_\eps(x)}v(y,t-\eps^2)\,dy\nonumber\\
&\quad +\frac{p-2}{p+n}\sup_{|\sigma|=1}\left[\frac{v(x+\eps\sigma,t-\eps^2)+v(x-\eps\sigma,t-\eps^2)}{2}\right]\nonumber\\
& =v(x,t)+\frac{\eps^2}{2(n+p)}(\mathcal{D}_p v(x,t)-2(n+p)v_t(x,t))+o(\eps^2),
\end{align}
which is valid for \emph{all} functions $v\in C^{2,1}(\Omega_T)$, see Theorem \ref{asymptotic}.

It turns out that the solution $u_\eps$ of DPP \eqref{dpp} is the value of the following time-dependent control problem. Let us denote  $\alpha= \frac{p-2}{p+n}, \beta= \frac{n+2}{p+n}$, and place a token at $(x_0,t_0) \in \Omega_T$. The controller tosses a biased coin with probabilities $\alpha$ and $\beta$. If she gets tails (with probability $\beta$), the game state moves according to the uniform probability density to a point $x_1 \in B_\eps (x_0)$. If the coin toss is heads (with probability $\alpha$), the controller chooses a unitary vector $\sigma \in \mathbb{R}^n$. The position of the token is then moved to $x_1= x_0 +\eps \sigma$ or $x_1=x_0 - \eps \sigma$ with equal probabilities. After this step, the position of the token is now at $(x_1,t_1)$, where $t_1= t_0 - \eps^2$. The game continues from $(x_1,t_1)$ according to the same rules yielding a sequence of game states
\begin{align*}
(x_0,t_0), (x_1,t_1), (x_2,t_2),...
\end{align*}
The game is stopped when the token is moved outside of $\Omega_T$ for the first time and we denote this point by $(x_\tau, t_\tau)$. The controller is then paid the amount $F(x_\tau, t_\tau)$. Naturally, the controller aims to maximize her payoff, and heuristically, the rules of the game can be read from the DPP \eqref{dpp}.

We remark that the scaling of the time derivative in equation \eqref{parab} is just a matter of convenience. For the equation $u_t=\mathcal{D}_p u$ we would define a game with the same rules as before, except that we would have $t_{j+1}=t_j-\frac{\eps^2}{2(n+p)}$ for every step in the game, see also Remark \ref{asympremark}.

This control problem has some similarities with two-player zero-sum \emph{tug-of-war} games, which were introduced by Peres, Schramm, Sheffield and Wilson \cite{peresssw09,peress08} and later studied from different perspectives, see e.g. \cite{AS12,manfredipr12,Lew18}. Time-dependent tug-of-war games, having connections to parabolic equations with the normalized $p$-Laplacian, were studied in \cite{MPR10,parviainenr16,Han18}, whereas two-player games for equations $u_t=\lambda_j(D^2 u)$, $j\in \{1,...,n\}$, were recently formulated in \cite{BER19}. For a deterministic game-theoretic approach to parabolic equations, we refer to \cite{KS10}.

This paper is organized as follows. In Section \ref{meanvalue} we prove the asymptotic mean value formula \eqref{asy}. In Section \ref{control} we show that the value of the control problem satisfies the DPP \eqref{dpp}. Finally, in Section \ref{conv} we show that value functions converge uniformly to the viscosity solution of \eqref{parab} when $\eps\rightarrow 0$.

\subsection*{Acknowledgements.} E.R. is supported by the Magnus Ehrnrooth Foundation. The authors would like to thank Peter Lindqvist and Tommi Brander for useful discussions.


\section{Asymptotic mean value formula}\label{meanvalue}

\begin{theorem}\label{asymptotic}
Let $v:\Omega_T\rightarrow \R$ be in $C^{2,1}(\Omega_T)$. Then it satisfies the asymptotic mean value formula \eqref{asy}.
\end{theorem}

\begin{proof}
Averaging the Taylor expansion 
\begin{align*}
v(y,t-\eps^2)&=v(x,t)+\langle D v(x,t),(y-x) \rangle+\frac12 \langle D^2 v(x,t)(y-x),(y-x) \rangle\\
&\quad -\eps^2 v_t(x,t)+o(|y-x|^2+\eps^2)
\end{align*}
over the ball $B_\eps(x)$ and calculating
\[
\vint_{B_\eps(x)}\langle D v(x,t),(y-x)\rangle\,dy=0
\]
and
\[
\vint_{B_\eps(x)} \langle D^2 v(x,t)(y-x),(y-x) \rangle\,dy=\frac{\eps^2}{n+2}\Delta v(x,t),
\]
we obtain
\begin{align}\label{laplace}
&\vint_{B_\eps(x)}v(y,t-\eps^2)\,\text{d}y\nonumber\\
& \quad=v(x,t)+ \frac{\eps^2}{2(n+2)}\Delta v(x,t)-\eps^2v_t(x,t)+o(\eps^2).
\end{align}

Next we take an arbitrary unit vector $\sigma$ and write the Taylor expansions for $v(x+h,t-\eps^2)$ with $h=\eps\sigma$ and $h=-\eps\sigma$ to obtain
\begin{align*}
v(x+\eps\sigma,t-\eps^2)&=v(x,t)+\langle D v(x,t),\eps\sigma \rangle+\frac12 \langle D^2 v(x,t)\eps\sigma,\eps\sigma \rangle\\
&\quad -\eps^2 v_t(x,t)+o(\eps^2),
\end{align*}
\begin{align*}
v(x-\eps\sigma,t-\eps^2)&=v(x,t)-\langle D v(x,t),\eps\sigma \rangle+\frac12 \langle D^2 v(x,t)(-\eps\sigma),(-\eps\sigma) \rangle\\
&\quad -\eps^2 v_t(x,t)+o(\eps^2),
\end{align*}
which yield 
\begin{align*}
&\frac{v(x+\eps\sigma,t-\eps^2)+v(x-\eps\sigma,t-\eps^2)}{2}\\
&\quad =v(x,t)+\frac{\eps^2}{2}\langle D^2 v(x,t)\sigma,\sigma \rangle-\eps^2 v_t(x,t)+o(\eps^2).
\end{align*}
Taking the supremum over all $|\sigma|=1$ gives
\begin{align}\label{inf}
&\sup_{|\sigma|=1}\left[\frac{v(x+\eps\sigma,t-\eps^2)+v(x-\eps\sigma,t-\eps^2)}{2}\right]\nonumber\\
&\quad =v(x,t)+\frac{\eps^2}{2}\lambda_n-\eps^2 v_t(x,t)+o(\eps^2).
\end{align}

By multiplying equations \eqref{laplace} and \eqref{inf} by $\frac{n+2}{p+n}$ and $\frac{p-2}{p+n}$ respectively, we get
\begin{align*}
&\frac{n+2}{p+n}\vint_{B_\eps(x)}v(y,t-\eps^2)\,dy\\
&\quad +\frac{p-2}{p+n}\sup_{|\sigma|=1}\left[\frac{v(x+\eps\sigma,t-\eps^2)+v(x-\eps\sigma,t-\eps^2)}{2}\right]\\
& =v(x,t)+\frac{\eps^2}{2(n+p)}(\mathcal{D}_p v(x,t)-2(n+p)v_t(x,t))+o(\eps^2).\qedhere
\end{align*}
\end{proof}

Next we define viscosity solutions for equation \eqref{parab}.

\begin{definition}
\label{def:viscosity-solution}
An upper semicontinuous function $u : \Om_T\to \R$ is a viscosity subsolution to equation $2(n+p)u_t=\mathcal{D}_p u$ in $\Omega_T$ if for all $(x_0,t_0)\in \Omega_T$ and $\phi \in C^2(\Om_T)$ such that
\begin{enumerate}
\item[i)]  $u(x_0, t_0) = \phi(x_0, t_0)$,
\item[ii)] $\phi(x, t) > u(x, t)$ for $(x, t) \in \Om_T,\ (x,t)\neq (x_0,t_0)$,
\end{enumerate}
it holds $2(n+p)\phi_t(x_0,t_0)\leq \mathcal{D}_p \phi(x_0,t_0)$. 

A lower semicontinuous function $u : \Om_T\to \R$ is a viscosity supersolution to equation $2(n+p)u_t=\mathcal{D}_p u$ in $\Omega_T$ if for all $(x_0,t_0)\in \Omega_T$ and $\phi \in C^2(\Om_T)$ such that
\begin{enumerate}
\item[i)]  $u(x_0, t_0) = \phi(x_0, t_0)$,
\item[ii)] $\phi(x, t) < u(x, t)$ for $(x, t) \in \Om_T,\ (x,t)\neq (x_0,t_0)$,
\end{enumerate}
it holds $2(n+p)\phi_t(x_0,t_0)\geq \mathcal{D}_p \phi(x_0,t_0)$. 

A continuous function $u : \Om_T\to \R$ is a viscosity solution to equation $2(n+p)u_t=\mathcal{D}_p u$ in $\Omega_T$ if it is both a subsolution and a supersolution.
\end{definition}

Because viscosity solutions of \eqref{parab} are in $C^{2+\alpha,\frac{2+\alpha}{2}}(\Omega_T)$ for some $\alpha>0$ (see Section \ref{intro}), we get the following corollary.

\begin{corollary}
Let $u$ be a viscosity solutions of \eqref{parab}. Then it satisfies an asymptotic mean value formula
\begin{align}\label{asymp}
u(x,t)&=\frac{n+2}{p+n}\vint_{B_\eps(x)}u(y,t-\eps^2)\,dy\nonumber\\
& \quad +\frac{p-2}{p+n}\sup_{|\sigma|=1}\left[\frac{u(x+\eps\sigma,t-\eps^2)+u(x-\eps\sigma,t-\eps^2)}{2}\right]+o(\eps^2).
\end{align}
\end{corollary}

\begin{remark}\label{asympremark}
Our scaling of the time variable is for convenience. The same idea would give for viscosity solutions of  
\begin{align*}
u_t = \mathcal{D}_p u
\end{align*}
an asymptotic mean value formula 
\begin{align*}
&u(x,t)=\frac{n+2}{p+n}\vint_{B_\eps(x)}u(y,t-\frac{\eps^2}{2(n+p)})\,dy\nonumber\\
& +\frac{p-2}{p+n}\sup_{|\sigma|=1}\left[\frac{u(x+\eps\sigma,t-\frac{\eps^2}{2(n+p)})+u(x-\eps\sigma,t-\frac{\eps^2}{2(n+p)})}{2}\right]+o(\eps^2).
\end{align*}
\end{remark}

\section{Control problem formulation}\label{control}

In this section we show that the value of the control problem described in Section \ref{intro} satisfies the DPP \eqref{dpp}. Since the game token may be placed outside of $\overline{\Omega}_T$, we denote the compact parabolic boundary strip of width $\eps>0$ by 
\begin{align*}
\Gamma_\eps=\left(S_\eps \times \big[-\eps^2,0\big]\right) \cup \left(\Omega \times \big [-\eps^2, 0\big] \right),
\end{align*}
where
\begin{align*}
S_\eps = \left \{  x\in \mathbb{R}^n \setminus \Omega\, :\, \text{dist}(x, \partial \Omega) \leq \eps   \right \}.
\end{align*}
Throughout this section, we are given a continuous function $$F: \Gamma_\eps \rightarrow \mathbb{R}.$$ Our control problem with the payoff $F$ was formulated in Section \ref{intro}. The process is stopped when the token hits the boundary strip $\Gamma_\eps$ for the first time at, say $(x_\tau,t_\tau)\in \Gamma_\eps$, and then the controller earns the amount $F(x_\tau,t_\tau)$.

Next we define the stochastic vocabulary for the control problem. A \textit{strategy} is a rule which gives, at each step of the game, a direction $\sigma$,
\begin{align*}
S(t_0, x_0,x_1,...,x_k)= \sigma\in \mathbb{R}^n, \quad |\sigma|=1. 
\end{align*}
Here, $S$ is a Borel measurable function. Let $A \subset \Omega_T \cup \Gamma_\eps$ be a measurable set.
Given a sequence of token positions $(x_0,t_0), (x_1, t_1),..., (x_k,t_k)$ and a strategy $S$, the next position of the token is distributed according to the transition probability
\begin{align*}
\pi_S \left( (x_0,t_0), (x_1, t_1),..., (x_k,t_k), A\right) &= \beta \frac{\left | A \cap \left(  B_\eps (x_k) \times \{ t_k- \eps^2\}\right) \right| }{\left | B_\eps (x_k) \times \{t_k -\eps^2  \}  \right| } \\
&+ \frac{\alpha}{2}\delta_{(x_k+ \eps \sigma, t_k-\eps^2 )}(A) + \frac{\alpha}{2}\delta_{(x_k-\eps \sigma, t_k-\eps^2 )}(A)
\end{align*}
where in the first term we use the $n$-dimensional Lebesgue measure, and in the last terms $\delta_{(y,s)}(B)=1$ if $(y,s)\in B$ and 0 otherwise.

For a starting point $(x_0,t_0)$, a strategy $S$ and the corresponding transition probabilities, we can use Kolmogorov's extension theorem to determine a unique probability measure $\mathbb{P}_S^{(x_0,t_0)}$ in the space of all game sequences denoted $H^\infty$. The expected payoff is then 
\begin{align*}
\mathbb{E}_S^{(x_0,t_0)} [F(x_\tau, t_\tau)] = \int_{H^\infty} F(x_\tau, t_\tau)\,  d\mathbb{P}_S^{(x_0,t_0)},
\end{align*} 
and the value of the game for the controller is
\begin{align*}
u^\eps (x_0,t_0)= \sup_S \mathbb{E}_S^{(x_0,t_0)} [F(x_\tau, t_\tau)].
\end{align*}
Since $F$ is bounded and 
\begin{align*}
\tau \leq \frac{T}{\eps^2}+1,
\end{align*}
the value of the game is well defined. From the definition we immediately get the following comparison principle.
\begin{proposition}\label{compprinciple}
Fix $\eps>0$. Let $u^\eps$ be the value of the game with the payoff $F_1$, and $v^\eps$ the value of the game with the payoff $F_2$. Assume that $F_1\geq F_2$ on $\Gamma_\eps$. Then $u^\eps\geq v^\eps$ in $\Omega_T$.
\end{proposition}
Our aim is to show that the value function $u^\eps$ satisfies the DPP with the boundary data $F$.

\begin{definition}
A Borel measurable function $u_\eps$ satisfies the dynamic programming principle, abbreviated DPP, in $\Omega_T$, with the boundary data $F$, if
\begin{align*}
u_\eps (x,t) &= \frac{n+2}{p+n}\vint_{B_\eps(x)}u_\eps(y,t-\eps^2 )\,dy\\
&\quad +\frac{p-2}{p+n}\sup_{|\sigma|=1}\left[\frac{u_\eps(x+\eps\sigma,t-\eps^2 )+u_\eps(x-\eps\sigma,t-\eps^2 )}{2}\right] \quad \text{in} \, \, \Omega_T \\
u_\eps (x,t) &= F(x,t) \quad \text{on} \, \,  \Gamma_\eps.
\end{align*}
\label{DPP}
\end{definition}
\begin{lemma}
There is a unique Borel measurable function $u_\eps$ satisfying the DPP. Moreover, $u_\eps$ is lower semi-continuous. 
\end{lemma}
\begin{proof}
 The existence and uniqueness of such a function $u_\eps$ can be seen from the following argument. Given $F$ on $\Gamma_\eps$, we can determine $u_\eps(x,t)$ for all $x \in \Omega$ and $0<t<\eps^2$. We want to continue this process, but we need to make sure that the function is lower semi-continuous or at least Borel measurable. The following argument is from personal communication with Brustad, Lindqvist, and Manfredi. In general, when $u$ is any bounded and lower semi-continuous function, then by using Fatou's lemma, 
\begin{align*}
& \frac{n+2}{p+n}\vint_{B_\eps(x)}u(y,t-\eps^2)\,dy\nonumber\\
	& \quad +\frac{p-2}{p+n}\sup_{|\sigma|=1}\left[\frac{u(x+\eps\sigma,t-\eps^2)+u(x-\eps\sigma,t-\eps^2)}{2}\right]
	\end{align*} 
	is again bounded and lower semi-continuous. This gives a lower semi-continuous function $u_\eps$ defined for all $x \in \Omega$ and $0<t<\eps^2$. Continuing this process until $t=T$ gives the desired function.

\end{proof}

\begin{lemma}
Let $u_\eps$ be the unique function satisfying the DPP of definition \ref{DPP} with the boundary data $F$ on $\Gamma_\eps$, and let $u^\eps$ be the value of the game with the payoff $F$. Then  
\begin{align*}
u_\eps = u^\eps.
\end{align*} 
\end{lemma}
\begin{proof}
Let $(x_0,t_0) \in \Omega_T$. We aim to show that $u_\eps(x_0,t_0) = u^\eps (x_0,t_0)$. Assume that the game starts at $(x_0,t_0)\in \Omega_T$.  

First we assume that the controller uses an arbitrary strategy $S$. Then we have for the function $u_\eps$ satisfying the DPP,
\begin{align*}
&\mathbb{E}_S^{(x_0,t_0)} [u_\eps(x_{k+1}, t_{k+1}) | (t_0,x_0,x_1,...,x_k)] = \beta \vint_{B_{\eps}(x_k)} u_\eps(y,t_k-\eps^2 ) \, dy  \\
& \quad + \alpha \frac{u_\eps (x_k+\eps \sigma, t_k - \eps^2 ) + u_\eps(x_k-\eps \sigma, t_k -\eps^2 )}{2} \\
& \leq \beta \vint_{B_{\eps}(x_k)} u_\eps(y,t_k-\eps^2) \, dy \\
& \quad + \alpha \sup_{|\sigma|=1}\left[\frac{u_\eps(x_k+\eps\sigma,t_k-\eps^2 )+u_\eps(x_k-\eps\sigma,t_k-\eps^2 )}{2}\right] \\
&= u_\eps (x_k,t_k).
\end{align*}
This shows that $M_k:=u_\eps (x_k,t_k)$ is a supermartingale, so $$\mathbb{E}_S^{(x_0,t_0)} [F(x_\tau, t_\tau) | (t_0,x_0,x_1,...,x_{\tau-1})] \leq u_\eps(x_0,t_0)$$ by the optimal stopping theorem. Hence
\begin{align*}
u^\eps (x_0,t_0) = \sup_S \mathbb{E}_S^{(x_0,t_0)} [F(x_\tau, t_\tau)] \leq u_\eps (x_0,t_0). 
\end{align*}

To prove the reverse inequality, we choose a strategy $S_0$ giving a corresponding $\sigma(x,t)$ for the controller that \textit{almost} maximizes $u_\eps(x,t)$. To be more precise, for arbitrary $\eta >0$, the controller chooses 
\begin{align*}
&\frac{u_\eps (x_k+\eps \sigma(x_k,t_k), t_k-\eps^2 )+u_\eps (x_k-\eps \sigma(x_k,t_k), t_k-\eps^2) }{2} \\ &\geq \sup_{|\sigma|=1}\left[\frac{u_\eps(x_k+\eps\sigma,t_k-\eps^2)+u_\eps(x_k-\eps\sigma,t_k-\eps^2)}{2}\right]- \eta 2^{-(k+1)}.
\end{align*}
The function $S_0$ can be taken to be a Borel function, see Lemma 3.4 in \cite{LM17}.

We obtain
\begin{align*}
&\mathbb{E}_{S_0}^{(x_0,t_0)} [u_\eps(x_{k+1}, t_{k+1}) - \eta 2^{-(k+1)} | (t_0,x_0,x_1,...,x_k)]\\
& \geq \beta \vint_{B_{\eps}(x_k)} u_\eps(y,t_k-\eps^2 ) \, dy  \\
& \quad+ \alpha  \sup_{|\sigma|=1}\left[\frac{u_\eps(x_k+\eps\sigma,t_k-\eps^2 )+u_\eps(x_k-\eps\sigma,t_k-\eps^2 )}{2}\right]\\
& \quad - \alpha \eta 2^{-(k+1)} - \eta 2^{-(k+1)} \\
&\geq u_\eps(x_k,t_k) - \eta 2^{-k}.
\end{align*}
Hence
\begin{align*}
M_k=u_\eps(x_k,t_k) - \eta 2^{-k}
\end{align*}
is a submartingale. Using the optimal stopping theorem for this submartingale we find
\begin{align*}
u^\eps (x_0,t_0) &=  \sup_S \mathbb{E}_S^{(x_0,t_0)} [F(x_\tau, t_\tau)] \geq \mathbb{E}_{S_0}^{(x_0,t_0)} [F(x_\tau, t_\tau)] \\
&\geq \mathbb{E}_{S_0}^{(x_0,t_0)} [u_\eps(x_\tau, t_\tau)- \eta 2^{-k}] \\
& \geq \mathbb{E}_{S_0}^{(x_0,t_0)} [u_\eps(x_0, t_0)- \eta 2^{-0}] = u_\eps (x_0,t_0) - \eta. 
\end{align*}
Since $\eta >0$ was arbitrary, this proves the lemma.  
\end{proof}

\section{Convergence to the viscosity solution}\label{conv}

In this section, we are given a continuous payoff function $F:\Gamma_1\rightarrow \R$. Our goal is to show that with this payoff, value functions of our game converge uniformly to the unique viscosity solution of 
\begin{equation}\label{Dirichlet}
\begin{split}
\begin{cases}
2(n+p)u_t=\mathcal{D}_p u  \quad  &\textrm{in}\quad \Omega_T,\\
u   =  F \quad  &\textrm{on}\quad  \partial_p \Omega_T.
\end{cases}
\end{split}
\end{equation}

We will make use of the following Arzel\'a-Ascoli-type lemma, which has been previously used e.g. in \cite{MPR10,parviainenr16,BER19}. We omit the proof, which is a modification of \cite[Lemma 4.2]{manfredipr12}.

\begin{lemma}\label{Ascoli}
Let $\left\{f_\eps:\overline{\Omega}_T\rightarrow \R\right\}_{\eps\in (0,1)}$ be a uniformly bounded family of functions such that for a given $\eta>0$, there are constants $r_0$ and $\eps_0$ such that for every $\eps<\eps_0$ and any $(x,t),(y,s)\in \overline{\Omega}_T$ with 
\[
|(x,t)-(y,s)|<r_0,
\]
it holds
\[
\left|f_\eps(x,t)-f_\eps(y,s)\right|<\eta. 
\]
Then there exists a uniformly continuous function $f:\overline{\Omega}_T\rightarrow \R$ and a subsequence, still denoted by $(f_\eps)$, such that $f_\eps\rightarrow f$ uniformly in $\overline{\Omega}_T$ as $\eps\rightarrow 0$.
\end{lemma}

For the next lemma, we assume that the domain $\Omega$ satisfies a \emph{uniform exterior sphere condition}. That is, we assume that there is $\delta>0$ such that for any $y\in \partial \Omega$, there is an open ball $B_\delta\subset \R^n\setminus \Omega$ with the radius $\delta$ so that $\overline B_\delta\cap \overline \Omega=\{y\}$. 

\begin{lemma}
The family $\{u_\eps\}_{\eps\in (0,1)}$ of value functions of the game satisfies the assumptions of Lemma \ref{Ascoli}.
\end{lemma}

\begin{proof}
Since $|u_\eps(x,t)|\leq \max_{\Gamma_1} |F|$ for all $(x,t)\in \overline{\Omega}_T$ and $\eps\in (0,1)$, the family $\{u_\eps\}_{\eps\in (0,1)}$ is uniformly bounded.

Fix $\eta>0$. Since the payoff function $F$ is uniformly continuous on $\Gamma_1$, there is $\gamma>0$ so that when $(x,t),(y,s)\in \Gamma_1$ with $|(x,t)-(y,s)|<\gamma$, it holds $|F(x,t)-F(y,s)|<\frac \eta2$. We prove the asymptotic equicontinuity of the family $\{u_\eps\}_{\eps\in (0,1)}$ in four steps. In all steps we have $\eps<\eps_0$ and $|(x,t)-(y,s)|<r_0$. The precise choices of $\eps_0$ and $r_0$ clarify during the proof. We will denote by $C_1,C_2,...$ constants larger than 1 which may depend only on $n,\delta,$ and the diameter of $\Omega$.

\subsection*{Step 1} If $(x,t),(y,s)\in \partial_p \Omega_T$, then $$|u_\eps(x,t)-u_\eps(y,s)|=|F(x,t)-F(y,s)|<\eta $$ when $r_0<\gamma$.

\subsection*{Step 2} Suppose that $(x,t)\in \Omega_T$ and $(y,0)\in \Gamma_\eps$. Let us start the game from $(x_0,t_0)=(x,t)$ with an arbitrary strategy $S$. We obtain 
\begin{align*}
&\mathbb{E}^{(x_0,t_0)}_S[|x_k-x_0|^2\, |\, (t_0,x_0,...,x_{k-1})]\\
& =\frac \alpha2(|(x_{k-1}+\sigma\eps)-x_0|^2+|(x_{k-1}-\sigma\eps)-x_0|^2)+\beta\vint_{B_\eps(x_{k-1})}|y-x_0|^2\,dy\\
& \leq \alpha(|x_{k-1}-x_0|^2+\eps^2)+\beta(|x_{k-1}-x_0|^2+C_1\eps^2)\\
& \leq |x_{k-1}-x_0|^2+C_1\eps^2.
\end{align*}
Hence, $$M_k:=|x_{k}-x_0|^2-C_1k\eps^2 $$ is a supermartingale, and the optimal stopping theorem gives
\begin{align*}
\mathbb{E}^{(x_0,t_0)}_{S}[|x_\tau-x_0|^2]\leq |x_0-x_0|^2+C_1\eps^2\mathbb{E}^{(x_0,t_0)}_{S}[\tau]\leq C_1(r_0+\eps_0^2).
\end{align*}
Here, we used the fact that the stopping time $\tau \leq \frac{t_0}{\eps^2 } + 1$ for a game starting at $t_0$ and in this case $t_0 \leq r_0$.
Since this is true for all strategies, it holds $$\sup_S \mathbb{E}^{(x_0,t_0)}_{S}[|x_\tau-x_0|^2]\leq C_1(r_0+\eps_0^2),$$ which yields 
\begin{align*}
|u_\eps(x_0,t_0)-u_\eps(x_0,0)|=|\sup_S \mathbb{E}^{(x_0,t_0)}_{S}[F(x_\tau,t_\tau)]-F(x_0,0)|<\frac{\eta}{2},
\end{align*}
when $r_0,\eps_0$ are chosen so that $C_1(r_0+\eps_0^2)<\gamma^2$. 

The triangle inequality finishes the argument. Recalling that $(x_0,t_0)=(x,t)$, we have $$|u_\eps(x,t)-u_\eps(y,0)|\leq |u_\eps(x,t)-F(x,0)|+|F(x,0)-F(y,0)|<\eta. $$

\subsection*{Step 3} 

Suppose that $(x,t)\in \Omega_T$ and $(y,s)\in \partial_p \Omega_T$ with $y\in \partial \Omega$. Since the domain $\Omega$ satisfies the uniform exterior sphere condition with $\delta$, there is a ball $B_\delta(z)\subset \R^n\setminus \Omega$ with $\partial B_\delta(z)\cap \overline{\Omega}=\{y\}$. 

We use a barrier argument. In an annulus of $\R^n$, define a function $w$ as
\begin{displaymath}
\left\{ \begin{array}{ll}
w(x)=-a|x-z|^2-b|x-z|^{-\xi}+c & \textrm{in}\ B_{R}(z)\setminus \overline{B}_{\delta}(z),\\
w=0 & \text{on}\ \partial B_{\delta}(z),\\
\frac{\partial w}{\partial \nu}=0 & \text{on}\ \partial B_{R}(z), 
\end{array} \right.
\end{displaymath}
where $\frac{\partial w}{\partial \nu}$ is the normal derivative, and $R$ is chosen so that $\Omega\subset B_R(z)$. The exponent $\xi=n+p-4>0$, since $p>2$ and we may assume that $n\geq 2$ (1-dimensional case is essentially a random walk in an open interval). The positive constants $a,b,c$ are specified below. The function $w$ satisfies $$\Delta w(x)=-2an+b\xi n|x-z|^{-\xi-2}-b\xi(\xi+2)|x-z|^{-\xi-2}, $$ $$\lambda_n(D^2 w(x))=-2a+b\xi|x-z|^{-\xi-2},$$ hence
\begin{equation}\label{poisson}
\mathcal{D}_p w=-2a(n+p-2)\quad \textrm{in}\ B_{R}(z)\setminus \overline{B}_{\delta}(z),
\end{equation}
and it can be extended as a solution to the same equations in $B_{R+\eps}(z)\setminus \overline{B}_{\delta-\eps}(z)$ so that equation \eqref{poisson} holds also near the boundaries. It satisfies an estimate

\[
w(x)\leq C_2(R/\delta)\dist(\partial B_\delta (z),x)+o(1)
\]
for any $x\in B_R(z)\setminus B_\delta (z)$. Here $o(1)\rightarrow 0$ when $\eps\rightarrow 0$.

Let us consider for a moment an elliptic game starting at $x_0=x$ and played by the rules of our game without a time-dependence in the annulus $B_R(z)\setminus \overline B_\delta(z)$, with a special rule that if we are at, say $x_k$, a possible random move is chosen from $B_\eps(x_k)\cap B_R(z)$ according to the uniform probability density, and also the controller cannot exit $B_R(z)$. The game ends when the token enters the ball $\overline{B}_\delta(z)$. Because of the random moves, the game ends almost surely in a finite time. Define a stopping time for this game as $\tau^*$,
\[
\tau^*=\inf\{k\ :\ x_k\in \overline{B}_\delta (z)\}.
\] 
Let $S$ be an arbitrary strategy for the controller. The Taylor expansion for $w$ gives 
\begin{align*}
&\frac12 (w(x_{k-1}+\eps\sigma)+w(x_{k-1}-\eps\sigma))\\
& \quad =w(x_{k-1})+\frac12\eps^2\langle D^2w(x_{k-1})\sigma,\sigma\rangle+o(\eps^2)\\
& \quad \leq w(x_{k-1})+\frac12\eps^2\lambda_n(D^2w(x_{k-1}))+o(\eps^2),
\end{align*}
since the first order terms vanish, $$\langle D w(x_{k-1}),\eps\sigma \rangle+\langle D w(x_{k-1}),-\eps\sigma \rangle=0. $$ Moreover, since $w$ is radially increasing, it holds 
\[
\vint_{B_\eps(x_{k-1})\cap B_R(z)}w(y)\,\text{d}y\leq w(x_{k-1})+ \frac{\eps^2}{2(n+2)}\Delta w(x_{k-1})+o(\eps^2).
\]
By choosing the constant $a$ properly, $$M_k:=w(x_k)+k\eps^2$$ is a supermartingale. Indeed, we have
\begin{align*}
\mathbb{E}_S^{x_0}[M_k\, |\, x_0,...,x_{k-1}]&=\frac{\alpha}{2}(w(x_{k-1}+\eps\sigma)+w(x_{k-1}-\eps\sigma))\\
& \quad+\beta\vint_{B_\eps(x_{k-1})\cap B_R(z)}w(y)\,dy+k\eps^2\\
& \leq w(x_{k-1})+\frac{\eps^2}{2(p+n)}\mathcal{D}_p w(x_{k-1})+k\eps^2+o(\eps^2)\\
& =w(x_{k-1})-\frac{n+p-2}{n+p}a\eps^2+k\eps^2+o(\eps^2)\\
& \leq w(x_{k-1})+(k-1)\eps^2,
\end{align*}
by choosing for example $a=2\frac{n+p}{n+p-2}$ and assuming that $o(\eps^2)<\eps^2$. The choice of $a$ determines the other constants $b$ and $c$: The Neumann and Dirichlet boundary conditions of the barrier function $w$ are satisfied by choosing $b=(2a/\xi)R^{\xi+2}$ and $c=a\delta^2+b\delta^{-\xi}$.

By the optimal stopping theorem, we have $$\mathbb{E}_S^{x_0}[w(x_{\tau*})+\tau^{*}\eps^2]\leq w(x_0), $$ that is, 
\begin{align*}
\mathbb{E}_S^{x_0}[\tau^*]\leq \frac{w(x_0)}{\eps^2}\leq \frac{C_2 (R/\delta)\dist(\partial B_\delta (z),x_0)+o(1)}{\eps^2},
\end{align*}
where we used $|\mathbb{E}_S^{x_0}[w(x_{\tau*})]|\leq o(1)$.

Now we come back to our game, starting at $(x_0,t_0)=(x,t)$, again with an arbitrary strategy $S$. Since it holds $|x_0-y|\geq \dist(\partial B_\delta (z),x_0)$, for the stopping time of our game we now have an estimate 
\begin{align*}
\mathbb{E}_S^{(x_0,t_0)}[\tau]&\leq \mathbb{E}_S^{(x_0,t_0)}[\tau^*]\\
& \leq \frac{C_2(R/\delta)\dist(\partial B_\delta (z),x_0)+o(1)}{\eps^2}\\
& \leq \frac{C_2(R/\delta)|x_0-y|+o(1)}{\eps^2}.
\end{align*}
By using the same martingale argument as in Step 2 but replacing $x_0$ by $y$, we have
\begin{align*}
\mathbb{E}^{(x_0,t_0)}_{S}[|x_\tau-y|^2]&\leq |x_0-y|^2+C_1\eps^2\mathbb{E}^{(x_0,t_0)}_{S}[\tau]\\
& \leq  |x_0-y|^2+C_1\eps^2 \frac{C_2(R/\delta)|x_0-y|+o(1)}{\eps^2}\\
& \leq |x_0-y|^2+C_3(|x_0-y|+o(1))\\
& <r_0^2+C_3(r_0+o(1))<\left(\frac{\gamma}{2}\right)^2,
\end{align*}
when $\eps_0,r_0$ are chosen so that $C_3(r_0+o(1))<\left(\frac{\gamma}{4}\right)^2$ and $r_0^2<\left(\frac{\gamma}{4}\right)^2$. This also gives 
\[
|\mathbb{E}^{(x_0,t_0)}_{S}[t_\tau]-t_0|<\left(\frac{\gamma}{4}\right)^2.
\]
Hence, we have
\begin{align*}
|u_\eps(x_0,t_0)-u_\eps(y,t_0)|=|\sup_S \mathbb{E}^{(x_0,t_0)}_{S}[F(x_\tau,t_\tau)]-F(y,t_0)|<\frac{\eta}{2},
\end{align*}
and recalling that $(x_0,t_0)=(x,t)$ the triangle inequality gives 
\[
|u_\eps(x,t)-u_\eps(y,s)|\leq |u_\eps(x,t)-F(y,t)|+|F(y,t)-F(y,s)|<\eta.
\]

\subsection*{Step 4} 
Finally, suppose that $(x,t),(y,s)\in \Omega_T$. This is an argument based on translation invariance and comparison principle. Let $r_0,\eps_0$ satisfy the conditions of the previous steps. Define an inner $\eps$-strip $I_\eps$ by 
\[
I_\eps:=\{(z,r)\in \overline{\Omega}_T\ :\ \dist((z,r),\partial_p \Omega_T)\leq r_0\}.
\]
If $(x,t)\in I_\eps$, there is a point $(x',t')\in \partial_p \Omega_T$ such that $|(x,t)-(x',t')|\leq r_0$. Then from the conclusions of the previous steps we obtain
\[
|u_\eps(x,t)-u_\eps(y,s)|\leq |u_\eps(x,t)-F(x',t')|+|F(x',t')-u_\eps(y,s)|<\eta.
\]
The argument is identical if $(y,s)\in I_\eps$, so it remains to study the case $(x,t),(y,s)\in \Omega_T\setminus I_\eps$. We may assume that $t\leq s$. Define functions $F_1,F_2$ on the strip $I_\eps$ as follows,
\[
F_1(z,r)=u_\eps(z-x+y,r-t+s)-\eta,\quad F_2(z,r)=u_\eps(z-x+y,r-t+s)+\eta.
\]
Then 
\[
F_1(z,r)\leq u_\eps(z,r)\leq F_2(z,r)
\]
for all $(z,r)\in I_\eps$. Let $u_\eps^1$ be the value function of the game in $\Omega_T\setminus I_\eps$ with the payoff $F_1$ on $I_\eps$, and $u_\eps^2$ the value function of the game in $\Omega_T\setminus I_\eps$ with the payoff $F_2$ on $I_\eps$. By the uniquess of the value function, we have for all $(z,r)\in \Omega_T\setminus I_\eps$
\begin{align*}
u_\eps^1(z,r)&=u_\eps(z-x+y,r-t+s)-\eta,\\
u_\eps^2(z,r)&=u_\eps(z-x+y,r-t+s)+\eta.
\end{align*}
By the comparison principle, see Proposition \ref{compprinciple}, we have 
\begin{align*}
u_\eps(x,t)&\geq u_\eps^1(x,t)=u_\eps(y,s)-\eta,\\
u_\eps(x,t)&\leq u_\eps^2(x,t)=u_\eps(y,s)+\eta.\qedhere
\end{align*}

\end{proof}

From the previous lemmas it follows that if $(u_{\eps_j})$ is a sequence of value functions with $\eps_j\rightarrow 0$ and $(u_{\eps_{j_k}})$ is an arbitrary subsequence, then this subsequence has a subsequence converging uniformly to $v$. Hence, the sequence $(u_{\eps_j})$ converges to $v$ uniformly, and we write $u_\eps\rightarrow v$ to simplify the notation. It remains to show that the function $v$ is the solution of \eqref{Dirichlet}.

\begin{theorem}\label{convergence}
The uniform limit $v= \lim_{\eps \rightarrow 0} u_\eps$ is the unique viscosity solution of \eqref{Dirichlet}. 
\end{theorem}

\begin{proof}
By uniqueness of viscosity solutions (see \cite{CIL92}), it is sufficient to show that $v$ is a viscosity solution of \eqref{Dirichlet}. To this end, let $\phi \in C^2$ touch $v$ from above at $(x_0,t_0)\in \Omega_T$, 
\begin{align*}
0= (v-\phi)(x_0,t_0) > (v-\phi)(x,t)
\end{align*}
for all $(x,t)$ close to $(x_0,t_0)$. From the definition of supremum, given $\delta_\eps>0$, there are points $(x_\eps, t_\eps)$ close to $(x_0,t_0)$ such that 
\begin{align*}
u_\eps (x_\eps, t_\eps) - \phi(x_\eps, t_\eps) \geq u_\eps (y,s) - \phi(y,s) - \delta_\eps
\end{align*}
for all $(y,s)$ in a neighborhood of $(x_\eps, t_\eps)$. Using the fact that $u_\eps \rightarrow v$ uniformly and $v-\phi$ is a continuous function with a maximum point at $(x_0,t_0)$, we see that $(x_\eps, t_\eps) \rightarrow (x_0,t_0)$ as $\eps \rightarrow 0$.

Since $\phi \in C^{2} (\Omega_T)$, Theorem \ref{asymptotic} gives
\begin{align*}
&\beta \vint_{B_\eps(x_\eps)}\phi(y,t_\eps-\eps^2 )\,dy\\
&\quad +\alpha\sup_{|\sigma|=1}\left[\frac{\phi(x_\eps+\eps\sigma,t_\eps-\eps^2 )+\phi(x_\eps-\eps\sigma,t_\eps-\eps^2 )}{2}\right]\\
& =\phi(x_\eps,t_\eps)+\frac{\eps^2}{2(n+p)}  (\mathcal{D}_p \phi(x_\eps,t_\eps)-2(n+p)\phi_t(x_\eps,t_\eps))+o(\eps^2).
\end{align*}
We can now estimate
\begin{align*}
&\beta \vint_{B_\eps(x_\eps)}u_\eps(y,t_\eps-\eps^2 )\,dy\\
&\quad +\alpha\sup_{|\sigma|=1}\left[\frac{u_\eps(x_\eps+\eps\sigma,t_\eps-\eps^2 )+u_\eps(x_\eps-\eps\sigma,t_\eps-\eps^2 )}{2}\right]\\ 
& \leq u_\eps(x_\eps, t_\eps) - \phi (x_\eps, t_\eps) + \delta_\eps  +\beta \vint_{B_\eps(x)} \phi  (y,t_\eps-\eps^2 )  \,dy\\
&\quad +\alpha\sup_{|\sigma|=1}\left[\frac{\phi(x_\eps+\eps\sigma,t_\eps-\eps^2)+\phi(x_\eps-\eps\sigma,t_\eps-\eps^2 ) }{2}\right]\\ 
& =u_\eps(x_\eps,t_\eps) + \delta_\eps  + \frac{\eps^2}{2(n+p)} (\mathcal{D}_p \phi(x_\eps,t_\eps)-2(n+p)\phi_t(x_\eps,t_\eps))    + o(\eps^2).
\end{align*}
As the function $u_\eps$ satisfies the DPP, we are left with
\begin{align*}
0 < \delta_\eps  + \frac{\eps^2}{2(n+p)}  (\mathcal{D}_p \phi(x_\eps,t_\eps)-2(n+p)\phi_t(x_\eps,t_\eps))    + o(\eps^2).
\end{align*}
Choose now $\delta_\eps =o(\eps^2)$. Dividing by $\eps^2 $ and letting $\eps \rightarrow 0$ gives $$2(n+p)\phi_t (x_0,t_0) \leq \mathcal D_{p} \phi (x_0,t_0),$$ which shows that $v$ is a viscosity subsolution. To show that $v$ is a viscosity supersolution is analogous. 
\end{proof}

\end{document}